\documentclass[11pt]{article}
\usepackage[top=3cm, bottom=3cm, outer=4cm, inner=4cm, heightrounded, marginparwidth=4.5cm, marginparsep=.5cm]{geometry}
\usepackage{amsmath,amssymb,amsthm}
\usepackage{graphicx}
\usepackage[T1]{fontenc}  
\usepackage{enumerate}  
\usepackage{float}   
\usepackage{tikz}
\usepackage{setspace}
\usetikzlibrary{matrix, arrows,calc}
\usepackage[english]{babel}
\usepackage[utf8]{inputenc}
\usepackage[colorinlistoftodos]{todonotes}
\usepackage[colorinlistoftodos]{todonotes}
\newtheorem{theorem}{Theorem}[section]

\newtheorem{lemma}[theorem]{Lemma}	      
\numberwithin{equation}{section}
\numberwithin{figure}{section}

\providecommand{\keywords}[1]
{
  \small	
  \textbf{\textit{Keywords---}} #1
}
\def\Re{{\rm Re}\,}
\def\Im{{\rm Im}\,}
\title{Octonionic Quadratic Equations}
\author{T.Kalpa Madhawa\\
\small Department of Mathematics\\
\small Southern Illinois University\\
}
\begin{document}
\maketitle
\begin{abstract} 
There are four division algebras over $\mathbb{R}$, namely real numbers, complex numbers, quaternions, and octonions. 
Lack of commutativity and associativity make it difficult to investigate algebraic and geometric properties of octonions. It does not make sense to ask, for example, whether the equation $x^2+1=0$ is solvable, 
without specifying the field in which we want the solutions to be lie. 
The equation $x^2+1=0$ has no solutions in $\mathbb{R}$,
which is to say, there are no real numbers satisfying this equation. 
On the other hand, there are complex numbers 
which do satisfy this equation in the field $\mathbb{C}$
of all complex numbers. 
How about if we extend the same idea to other 
two normed division algebras quaternions and octonions.
Liping Huang and Wasin So \cite{L.Huang} derive explicit formulas for computing the roots of quaternionic quadratic equations. We extend their work to octonionic case and solve monic left octonionic quadratic equation of the form $x^2+bx+c=0$, where $a,b$ are octonions in general.[ We called this form of quadratic equation as left octonion quadratic equation because we can consider $x^2+xb+c=0$ as a different case due to non-commutativity of octonions].
Finally, we represent the left spectrum of
$2\times2$ octonionic matrix as a set of solutions to a corresponding octonionic quadratic equation, which is an application of deriving explicit formulas for computing the roots of octonionic quadratic equations.
\end{abstract}
\keywords{octonionic quadratic equation, left spectrum.}
\section{Quaternions and Octonions}
\subsection{Algebra of quaternions}
Quaternions arose historically from Sir William Rowan Hamilton's attempts in the midnineteenth century to generalize complex numbers in some way that would applicable to three-dimensional space. \cite{Tevian} In complex numbers we have square root of -1 called $i$, what happens if we include another, independent, square root of -1? Call it $j$. Then the big question is, what is $ij$?
William Hamilton eventually proposed that $k=ij$ should be yet another square root of -1, and that the multiplication table should be cyclic, that is
\begin{equation*}
\begin{split}
ij=k=-ji,\\
jk=i=-kj, \\
ki=j=-ik,
\end{split}
\end{equation*}
\begin{figure}[h] 
  \centering
  \includegraphics[scale=.5]{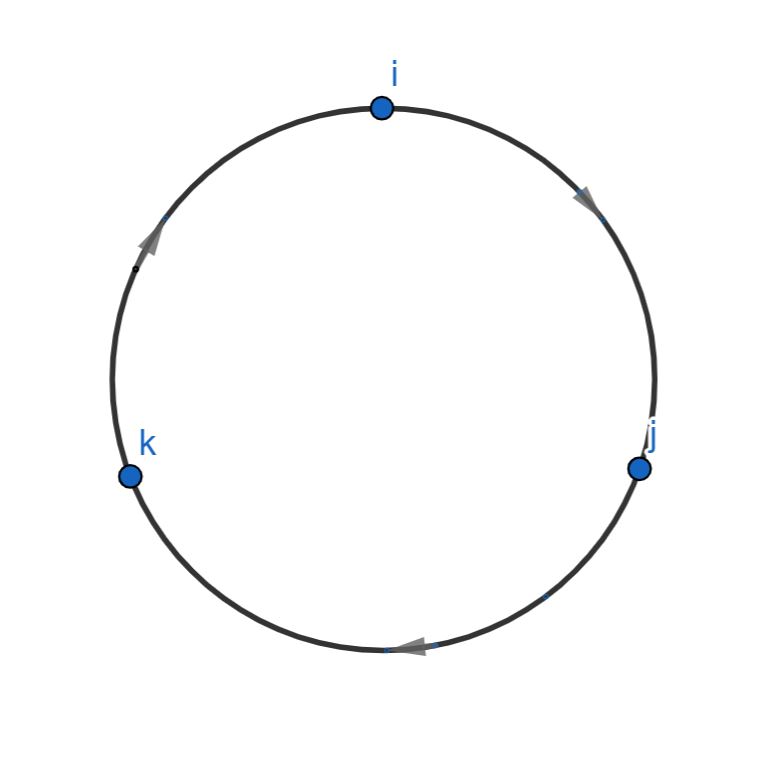}
  \caption{The quaternionic multiplication table }
  \label{fig:QM_}
\end{figure}

\newpage
\noindent We refer to $i$,$j$, and $k$ as imaginary quaternionic units. Notice that these units anticommute.
This multiplication table is shown schematically in Figure $1.1$ multiplying two of these quaternionic units together in the direction of the arrow yields the third; going against the arrow contributes an additional minus sign.
The quaternions are denoted by $\mathbb{H}$; the H is for "Hamilton", they are spanned by the identity element 1 and three imaginary units, that is , a quaternion $q$ can be represented as four real numbers $(q_1,q_2,q_3,q_4)$, usually written
\begin{equation}
q=q_1+q_2i+q_3j+q_4k
\end{equation}
which can be thought of as a point or vector in $\mathbb{R}^4$. since $(1.1)$ can be written in the form
\begin{equation}
         q = (q_1+q_2i)+(q_3+q_4i)j
\end{equation}
we see that a quaternion can be viewed as a pair of complex numbers, we can write
$\mathbb{H}=\mathbb{C}\oplus\mathbb{C}j$
in direct analogy to the construction of $\mathbb{C}$ from $\mathbb{R}$.
The quaternionic conjugate $\bar{q}$ of a quaternion $q$ is obtained via the(real) linear map which reverses the sign of each imaginary unit, so that
\begin{equation}
\bar{q}=q_1-q_2i-q_3j-q_4k
\end{equation}

\noindent if $q$ is given by $(1.1)$. Conjugation leads directly to the norm of quaternion $|q|$, defined by 

\begin{equation}
|q|^2=q\bar{q}=q_1^2+q_2^2+q_3^2+q_4^2.
\end{equation}

\noindent again, the only quaternion with norm zero is zero, and every nonzero quaternion has a unique inverse, namely

\begin{equation}
q^{-1}=\bar{q}/|q|^2
\end{equation}
uaternionic conjugation satisfies the identity $\overline{pq}=\bar{q}\bar{p}$ from which it follows that the norm satisfies $|pq|=|p||q|$.
\noindent Squaring both sides and expanding the result in terms of components yields the $4-$ sqaures rule,
\begin{equation}
\begin{split}
&(p_1q_1-p_2q_2-p_3q_3-p_4q_4)^2+(p_2q_1+p_1q_2-p_4q_3+p_3q_4)^2  \\
&+(p_3q_1+p_4q_2+p_1q_3-p_2q_4)^2+(p_4q_1-p_3q_2+p_2q_3+p_1q_4)^2 \\
&=(p_1^2+p_2^2+p_3^2+p_4^4)(q_1^2+q_2^2+q_3^2+q_4^2)
\end{split}
\end{equation}
\noindent which is not quite as obvious as the $2-$ squars rule. This identity implies that the quaternions form a normed division algebra, that is, not only are there inverses, but there are no zero divisors-if a product is zero, one of the factors must be zero.
It is important to realize that $\pm i,\pm j$, and $\pm k$ are not the only quaternionic square roots of $-1$, 

\newpage
\noindent any imaginary quaternion squares to a negative number, so it is only necessary to choose its norm to be one in order to get a square root of $-1$. The imaginary quaternions of norm one form a $2-$dimensional sphere($q_1=0$); in the above notation, this is the set of points
\begin{equation}
q_2^2+q_3^2+q_4^2=1
\end{equation}

\noindent any such unit imaginary quaternion $u$ can be used to construct a complex subalgebra of $\mathbb{H}$, which we will also denote by $\mathbb{C}$, namely

\begin{equation*}
\mathbb{C}={a+bu}
\end{equation*}

\noindent with $a,b \in \mathbb{R}$. Furthermore, we can use the identity to write 

\begin{equation*}
e^{u\theta}=\cos\theta+u\sin\theta
\end{equation*}

\noindent This means that any quaternion can be written in the form
$q=re^{u\theta}$\ where $r=|q|$\ and\ $u$\ denotes the direction of the imaginary part of $q.$

\newpage
\subsection{Algebra of octonions}

In analogy to the previous construction of $\mathbb{C}$ and $\mathbb{H}$, an octonion  $x$ can be thought of as a pair of quaternions, so that
\begin{equation*}
                   \mathbb{O}=\mathbb{H}\oplus \mathbb{H} \, l
\end{equation*}
we will denote $i$ times $l$ simply as $il$, and similarly with $j$ and $k$. It is easy to see that $l$,$il,jl,$ and $kl$ all square to $-1$; there are now seven independent imaginary units, and we could write

\begin{equation}
x=x_1+x_2i+x_3j+x_4k+x_5l+x_6il+x_7jl+x_8kl 
\ \ where\ 
x_m\in\mathbb{R} \quad \ for\ i=1,2,...,8.
\end{equation}
\noindent which can be thought of as a point or vector in $\mathbb{R}^8$. The real part of $x$ is just $x_1$; the imaginary part of $x$ is everything else. Algebraically, we could define
\begin{equation}
\Re(x) = (x+\bar{x})/2 
\end{equation}

\begin{equation}
\hbox{Im}(x) = (x-\bar{x})/2
\end{equation}

\noindent The imaginary part is differs slightly from the standard usage of these terms for complex numbers, where $\hbox{Im}(z)$ normally refers to a real numbers, the coefficient of $i$. This convention is not possible here, since the imaginary part has seven degrees of freedom, and can be thought of as a vector in $\mathbb{R}^7$.
The full multiplication table is summarized in Figure $1.2$ 
by means of the $7-$point projective plane. 
Each point corresponds to an imaginary unit.

\begin{figure}[h] 
  \centering
  \includegraphics[scale=.6]{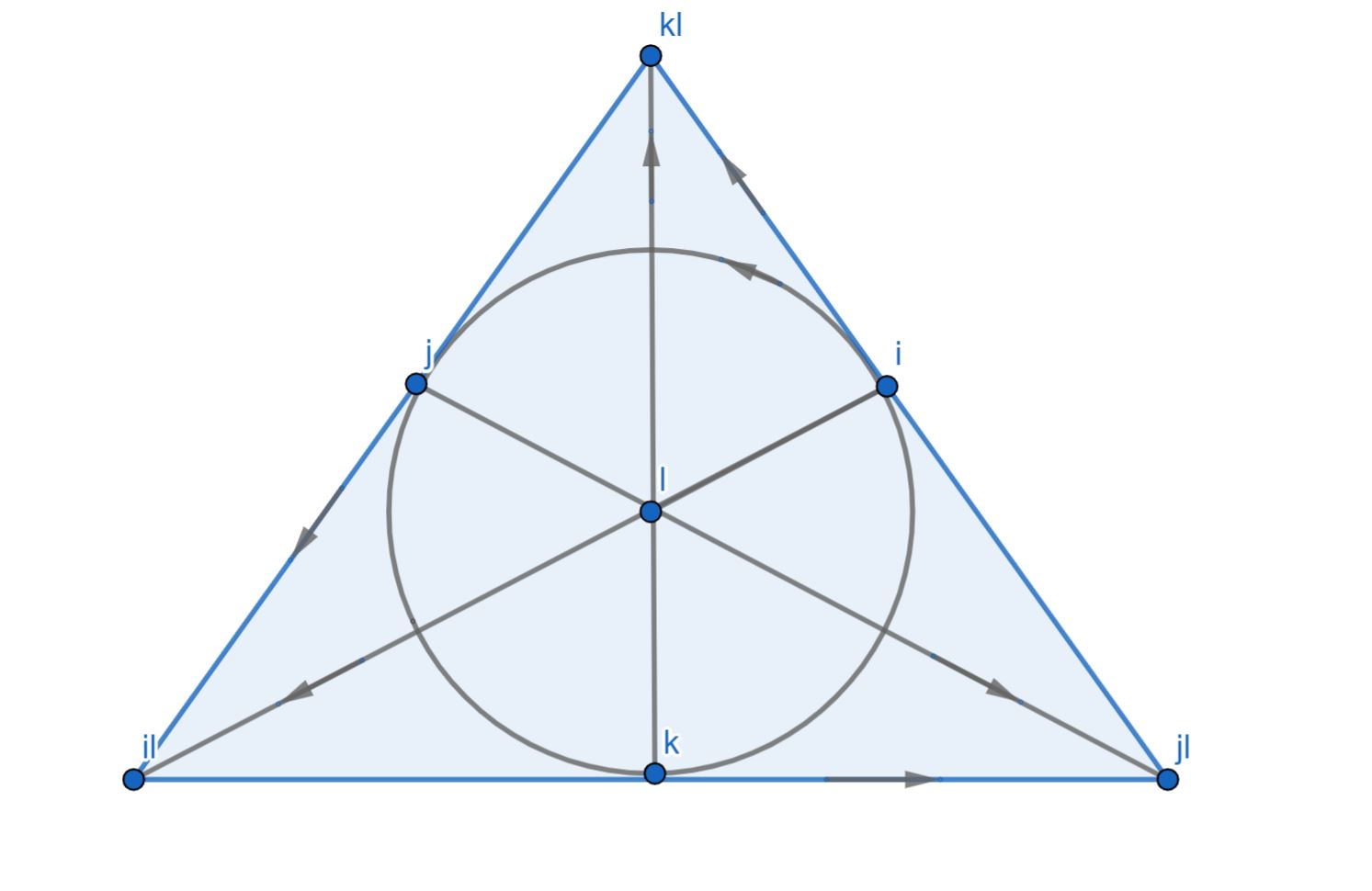}
  \caption{The octonionic multiplication table}
  \label{fig:oct2}
\end{figure}

\newpage
\noindent 
Each line corresponds to a quaternionic triple, much like $\{i,j,k\}$, 
with the arrow giving the orientation. 
For Example,

\begin{equation}
\label{eq:kl}
\begin{split}
&kl=kl\\
&lkl=k\\
&klk=l
\end{split}
\end{equation}
and each of these products anticommutes, that is, reversing the order contributes a minus sign.
We define the octonionic conjugate $\bar{x}$ of an octonion $x$ as the (real) linear map which reverses the sign of each imaginary unit. 
Thus,
\begin{equation}
\bar{x}=x_1-x_2i-x_3j-x_4k-x_5l-x_6il-x_7jl-x_8kl
\end{equation}
if $x$  is given by $(1.8)$. Direct computation shows that
\begin{equation}
         \overline{xy}=\bar{y}\bar{x}
\end{equation}
The norm of an octonion $|x|$ is defined by
\begin{equation}
|x|^2=x\bar{x}=x_1^2+x_2^2+x_3^2+x_4^2+x_5^2+x_6^2+x_7^2+x_8^2
\end{equation}
Again, the only octonion with norm zero is zero, and every nonzero octonion has a unique inverse, namely

\begin{equation}
x^{-1} =\bar{x}/|x|^2
\end{equation}

\noindent as with the other division algebras, the norm satisfies the identity

\begin{equation}
|xy|=|x||y|
\end{equation}

\noindent writing out this expression in terms of components yields the $8-$squares rule, which is no longer at all obvious. The octonions therefore also form a normed division algebra.\\
A remarkable property of the octonions is that they are not associative, For Example

\begin{equation}
\begin{split}
&(ij)(l)=+(k)(l)=+kl\\
&(i)(jl)=(i)(jl)=-kl
\end{split}
\end{equation}

\newpage
\section{Octonionic quadratic formulas.}

We will start with two lemmas.

\begin{lemma}
\label{prop:1}
Let $B, E,$ and $D$ be real numbers such that
\begin{enumerate}
    \item $D\neq 0$, and
    \item $B<0$ implies $B^2<4E$
\end{enumerate}
Then the cubic equation
\begin{equation*}
    y^3+2By^2+(B^2-4E)y-D^2=0
\end{equation*}
has exactly one positive solution.
\end{lemma}

\begin{proof}
Let
\begin{equation*}
    f(y)= y^3+2By^2+(B^2-4E)y-D^2=0
\end{equation*}
Note that $f(0)=-D^2<0$ and $\lim_{y \to +\infty} f(y)=+\infty$.
According to the intermediate value Theorem:  if $f$ is a continuous function over an interval [a,b], then $f$ takes all values between $f(a)$ and $f(b)$.
Since above cubic polynomial is a continuous function, its graph must intersect the x-axis at some finite points greater than zero. So the equation has at least one positive root. 
Now let us prove that $f$ has only one positive root.
\\
Suppose that $f$ has three real roots, $r_1,r_2$ and $r_3$. Take $r_1>0$ be the positive root we found above. We must show that $r_2,r_3<0$. Then we have the result.
\\
we know,
\begin{equation*}
    r_1.r_2.r_3=D^2>0
\end{equation*}
this implies the product of $r_2$ and $r_3$ is positive. Therefore, $r_2$ and $r_3$ should be in same sign( both positive or negative).
Let's assuming $r_2,r_3>0$, 
\\
if 
$B<0$ implies $r_1+r_2+r_3=-2B>0$ 
But
\begin{equation*}
    r_1.r_2+r_1.r_3+r_2.r_3=B^2-4E
\end{equation*}
and should be positive. This is a contradiction due to part $2$ of $\textbf{Lemma 2.1}.$ thus, we have only one positive solution to $f(y)=0$. 
\end{proof}

\newpage 
\begin{lemma}
\label{prop:2}
Let $B, E,$ and $D$ be real numbers such that
\begin{enumerate}
    \item $E\geq 0$, and
    \item $B<0$ implies $B^2<4E$
\end{enumerate}
Then the real system
\begin{equation}
   N^2-(B+T^2)N+E=0,
\end{equation}
\begin{equation}
      T^3+(B-2N)T+D=0, 
\end{equation}
has at most two solutions $(T,N)$ satisfying $T\in\mathbb{R}$ and $N\geq0$ as follows.
\begin{enumerate}
\item $T=0, N=(B\pm\sqrt{B^2-4E)}/2$ provided that $D=0, B^2\geq 4E.$
\item $T=\pm\sqrt{2\sqrt{E}-B}, N=\sqrt{E}$ provided that $D=0, B^2 < 4E.$
\item $T=\pm\sqrt{z}, N=(T^3+BT+D)/2T$ provided that $D\neq 0$ and $z$ is the 
unique positive root of the real polynomial\ \ $z^3+2Bz^2+(B^2-4E)z-D^2$.
\end{enumerate}
\end{lemma}
\begin{proof}
(a)\ if $D=0$ by (2.2) then we have $T=0$ and (2.1) gives $N=(B\pm\sqrt{B^2-4E)}/2$ 
\\
(b)\ If $D=0$ and $B^2<4E$, by (2.2)we have
\begin{equation*}
    T(T^2-(2N-B))=0\ \ \textrm{Note that}\ \ 2N-B>0
\end{equation*}
Therefore,
\begin{equation*}
    T=\pm\sqrt{2N-B}
\end{equation*}
By (2.1) we have,
\begin{equation*}
    N^2-(B+2N-B)N+E=0 
\end{equation*}
\begin{equation*}
    N^2=E\ \ \textrm{implies}\ \ N=\sqrt{E}
\end{equation*}
Hence,
\begin{equation*}
    T=\pm\sqrt{2\sqrt{E}-B}\ \ \textrm{and}\ \ N=\sqrt{E}
\end{equation*}
(c) If $D\neq0$ from (2.2) $N=(T^3+BT+D)/2T$ plug this $N$ value to (2.1) we have
\begin{equation*}
    \Big(\frac{T^3+BT+D}{2T}\Big)^2-(B+T^2) \Big(\frac{T^3+BT+D}{2T}\Big)+E=0,
\end{equation*}
\begin{equation*}
    \Big(T^3+BT+D\Big)^2-2T(B+T^2) \Big(T^3+BT+D\Big)+4T^2E=0,
\end{equation*}
\begin{equation*}
    \Big(T^3+BT+D\Big)^2-2T(B+T^2) \Big(T^3+BT+D\Big)+4T^2E=0,
\end{equation*}
\begin{equation*}
    T^6+2BT^4+(B^2-4E)T^2-D^2=0,
\end{equation*}
Let $T=\pm\sqrt{z}$ Then, by $\textbf{Lemma 2.1}.$ we have $z$ as unique positive root satisfies cubic equation
$z^3+2Bz^2+(B^2-4E)z-D^2$.
\end{proof}
\newpage
\begin{theorem} The solutions of left octonionic quadratic equation $x^2+bx+c = 0$ can be obtained by formulas according to the following cases.
\\
\textbf{case 1.} 
 if $b,c\in\mathbb{R}$ and $b^2\geq 4c$, then
\begin{equation*}
\label{eq:my x}
    x=\frac{-b\pm\sqrt{b^2-4c}}{2}
\end{equation*}
\textbf{case 2.}
if $b,c \in \mathbb{R}$ and $b^2<4c,$ then
\begin{equation*}
    x=\frac{1}{2}(-b+\alpha \textbf{i}+\beta\textbf{j}+\gamma \textbf{k}+\delta\textbf{l}+p\textbf{il}+ q\textbf{jl}+r\textbf{kl})
\end{equation*}
\begin{equation*}
\ \ \textrm{where} \ \    \alpha^2+\beta^2+\gamma^2+\delta^2+p^2+q^2+r^2=4c-b^2
    \ \ \textrm{and} \ \  \alpha,\beta,\gamma,\delta,p,q,r\in\mathbb{R}
\end{equation*}
\textbf{case 3.} if $b\in\mathbb{R}$ and $c\notin\mathbb{R}$, then
\begin{equation*}
\hspace{2cm}
    x={\frac{-b}{2}}\pm {\frac{\rho}{2}}\mp { \frac{c_1}{\rho}}\textbf{i}\mp{\frac{c_2}{\rho}}\textbf{j}\mp{\frac{c_3}{\rho}}\textbf{k}\mp{\frac{c_4}{\rho}}\textbf{l}\mp{\frac{c_5}{\rho}}\textbf{il}\mp{\frac{c_6}{\rho}}\textbf{jl}\mp {\frac{c_7}{\rho}}\textbf{kl}
\end{equation*}
\vspace{-2mm}
\begin{equation*}
\textrm{where} \hspace{1mm} c=c_0+c_1\textbf{i}+c_2\textbf{j}+c_3\textbf{k}+c_4\textbf{l}+c_5\textbf{il}+c_6\textbf{jl}+c_7\textbf{kl}  \hspace{2mm} \textrm{and} \hspace{2mm}
\end{equation*}
\begin{equation*}
\rho=\sqrt{{\bigg(b^2-4c_0+\sqrt{(b^2-4c_0)^2 +16(c_1^2+c_2^2+c_3^2+c_4^2+c_5^2+c_6^2+c_7^2)}\bigg)}/2}
\end{equation*}
\textbf{case 4.} if $b\notin\mathbb{R}$, then
\begin{equation*}
x=\frac{-\Re b}{2}-(b'+T)^{-1}(c'-N),
\end{equation*}
where
\ $$b'=b-\Re\ b=\textrm{Im}\ b,\  c'=c-\big((\Re b)/2)(b-(\Re b)/2\big),\ \textrm{and}\  (T,N)\  \textrm{is}\ \textrm{chosen} \  \textrm{as}\ \textrm{follows}.$$
\begin{enumerate}
\item $T=0, N=(B\pm\sqrt{B^2-4E})/2$ provided that $D=0$, $B^2\geq 4E.$
\item $T=\pm\sqrt{2\sqrt{E}-B}, N=\sqrt{E}$ provided that $D=0, B^2 < 4E.$
\item $T=\pm\sqrt{z}, N=(T^3+BT+D)/2T$ provided that $D\neq 0$ and $z$ is the unique positive root of the real polynomial \  $z^3+2Bz^2+(B^2-4E)z-D^2$,
\\
where $B=b'\bar{b'}+c'+\bar{c'}=|b'|^2+2\Re c',\ E=c'\bar{c'}=|c'|^2 ,\ D=\bar{b'}c'+\bar{c'}b'=2\Re (\bar{b'}c')$, are real numbers.
\end{enumerate}
\end{theorem}
\newpage
\begin{proof}
\textbf{Case 1.}
$b,c\in\mathbb{R}$ and $b^2\geq 4c.$ Note that $x$ is a solution if and only if $q^{-1}xq$ is also a solution for $q\neq0,$ and hence, there are at most two solutions, both are real
\begin{equation*}
    \frac{-b\pm\sqrt{b^2-4c}}{2}
\end{equation*}
\noindent\textbf{Case 2.}
$b,c\in\mathbb{R}$ and $b^2<4c.$ Note that $x$ is a solution if and only if $q^{-1}xq$ is also a solution for $q\neq0,$ and there are at least two complex solutions
\begin{equation*}
    \frac{-b\pm\sqrt{4c-b^2} \textbf{i}}{2}
\end{equation*}
Hence, the solution set is
\begin{equation*}
    \Bigg\{ q^{-1}\Big(\frac{-b+\sqrt{4c-b^2}\textbf{i}}{2}\Big)q: q\neq0\Bigg\}
\end{equation*}
Let $R^2=4c-b^2>0$ and $q\in\mathbb{O}$
\begin{equation*}
    \Bigg\{ q^{-1}\Big(\frac{-b+R\textbf{i}}{2}\Big)q: q\neq0\Bigg\}
\end{equation*}
\begin{equation*}
    \Bigg\{\frac{-b+q^{-1}R\textbf{i}q}{2}: q\neq0\Bigg\}
\end{equation*}
\begin{equation*}
    \Bigg\{\frac{-b+Rq^{-1}\textbf{i}q}{2}: q\neq0\Bigg\}
\end{equation*}
\begin{equation*}
    Rq^{-1}iq=\alpha \textbf{i}+\beta\textbf{j}+\gamma \textbf{k}+\delta\textbf{l}+p\textbf{il}+ q\textbf{jl}+r\textbf{kl}\in\Im\mathbb{O}
\end{equation*}
and
\begin{equation*}
    \overline{Rq^{-1}iq}=-Rq^{-1}iq
\end{equation*}
Therefore,
\begin{equation*}
    |Rq^{-1}iq|=(Rq^{-1}iq)\overline{(Rq^{-1}iq)}=-(Rq^{-1}iq)(Rq^{-1}iq)=R^2=4c-b^2
\end{equation*}
implies,
\begin{equation*}
    \alpha^2+\beta^2+\gamma^2+\delta^2+p^2+q^2+r^2=4c-b^2   
\end{equation*}
Hence solution set is,
\begin{equation*}
 = \Bigg\{\frac{1}{2}(-b+\alpha \textbf{i}+\beta\textbf{j}+\gamma \textbf{k}+\delta\textbf{l}+p\textbf{il}+ q\textbf{jl}+r\textbf{kl}) :\alpha^2+\beta^2+\gamma^2+\delta^2+p^2+q^2+r^2=4c-b^2\Bigg\}
\end{equation*}
\noindent How we will describe having infinitely many solutions? 
\\ For Geometric view,
\\
Consider
\begin{equation*}
|x|^2=\frac{(-b)^2}{4}+\frac{\alpha^2}{4}+\frac{\beta^2}{4}+ \frac{\gamma^2}{4}+\frac{\delta^2}{4}+\frac{p^2}{4}+\frac{q^2}{4}+\frac{r^2}{4}=\frac{b^2+4c-b^2}{4}=c
\end{equation*}
\textbf{Conclusion}: Solutions are set of all points on $S^7$ ($7$ - sphere) in 8-dimension with norm equal to $\sqrt{c}$ (Note that $c\in\mathbb{R}).$

\newpage
\noindent \textbf{Case 3.}
$b\in\mathbb{R}$ and $c\notin\mathbb{R}.$ 
Let $$x=x_0+x_1\textbf{i}+x_2\textbf{j}+x_3\textbf{k}+x_4\textbf{l}+x_5\textbf{il}+x_6\textbf{jl}+x_7\textbf{kl}$$
and  $$c=c_0+c_1\textbf{i}+c_2\textbf{j}+c_3\textbf{k}+c_4\textbf{l}+c_5\textbf{il}+c_6\textbf{jl}+c_7\textbf{kl}$$ Then $x^2+bx+c=0$ becomes the real system
\begin{equation*}
    (x_0^2-x_1^2-x_2^2-x_3^2-x_4^2-x_5^2-x_6^2-x_7^2)+bx_0+c_0=0
\end{equation*}
\begin{equation*}
    \Big (x_0^2+\frac{b}{2}\Big)^2-x_1^2-x_2^2-x_3^2-x_4^2-x_5^2-x_6^2-x_7^2\big )=\frac{b^2}{4}-c_0
\end{equation*}
\begin{equation*}
    (2x_0+b)x_i=-c_i \ \ where \ \ i=1,2,3,4,5,6,7
\end{equation*}
Since $c$ is non-real $(2x_0+b)$ is non-zero
$$(2x_0+b)^2-4x_1^2-4x_2^2-4x_3^2-4x_4^2-4x_5^2-4x_6^2-4x_7^2=(b^2-4c_0)$$
$$(2x_0+b)^4-4c_1^2-4c_2^2-4c_3^2-4c_4^2-4c_5^2-4c_6^2-4c_7^2=(2x_0+b)^2(b^2-4c_0)$$
$$(2x_0+b)^4-(b^2-4c_0)(2x_0+b)^2+4(c_1^2+c_2^2+c_3^2+c_4^2+c_5^2+c_6^2+c_7^2)=0$$
\begin{equation*}
    (2x_0+b)^2=\frac{(b^2-4c_0)\pm\sqrt{(b^2-4c_0)^2+16(c_1^{2}+c_2^{2}+c_3^{2}+c_4^{2}+c_5^{2}+c_6^{2}+c_7^{2}})}{2}
\end{equation*}
\begin{equation*}
   (2x_0+b)=\pm\sqrt{\frac{\Big(b^2-4c_0+\sqrt{(b^2-4c_0)^2+16(c_1^{2}+c_2^{2}+c_3^{2}+c_4^{2}+c_5^{2}+c_6^{2}+c_7^{2}})\Big)}{2}}
\end{equation*}
\begin{equation*}
x_0=\frac{(-b\pm\rho)}{2} \ \  where \ \rho\neq 0
\end{equation*}
\begin{equation*}
x=\frac{-b}{2}\pm\frac{\rho}{2}\mp\frac{c_1}{\rho}\textbf{i}\mp\frac{c_2}{\rho}\textbf{j}\mp\frac{c_3}{\rho}\textbf{k}\mp\frac{c_4}{\rho}\textbf{l}\mp\frac{c_5}{\rho}\textbf{il}\mp\frac{c_6}{\rho}\textbf{jl}\mp\frac{c_7}{\rho}\textbf{kl}
\end{equation*}
\\
\noindent
\textbf{Case 4.} $b\notin\mathbb{R}.$ Rewrite the equation $x^2+bx+c=0$ as
\begin{equation*}
    y^2+b'y+c'=0,
\end{equation*}
where $y=x+(\Re b)/2$, $b'=b-\textrm{Re}b\notin\mathbb{R}$ 
and $c'=c-(\Re b/2)\big(b-(\Re b/2)\big)$. By \cite{Niven},
we observe that the solution of the quadratic equation $y^2+b'y+c'=0$ also satisfies
\begin{equation*}
    y^2-Ty+N=0,
\end{equation*}
where $N=\Bar{y}y\geq0$ and $T=y+\Bar{y}\in\mathbb{R}.$ Hence, $(b'+T)y+(c'-N)=0,$
 and so
\begin{equation*}
    y=-(b'+T)^{-1}(c'-N)
\end{equation*}
\\
because $T\in\mathbb{R}$ and $b'\notin\mathbb{R}$ implies that $b'+T\neq 0$
To solve for $T$ and $N$, we substitute $y$ back into the definitions $T=y+\bar{y}$ and $N=\bar{y}y$ and simplify to obtain the real system
\begin{equation*}
    N^2-(B+T^2)N+E=0,
\end{equation*}
\begin{equation*}
    T^3+(B-2N)T+D=0,
\end{equation*}
where $B=b'\bar{b'}+c'+\bar{c'}=|b'|^2+2\Re c', E=c'\bar{c'}=|c'|^2 , D=\bar{b'}c'+\bar{c'}b'=2\Re \bar{b'}c'$ are real numbers. 
Note that $E=|c'|^2\geq0.$
if $B<0,$ then $c'+\bar{c'}<0$ and $B^2-4E=|b'|^2B+|b'|^2(c'+\bar{c'})+(c'-\bar{c'})^2\leq0$ because of the face that $(c'-\bar{c'})^2\leq 0.$ It follows that $B^2-4E<0,$ otherwise $B^2-4E=0$ and so
$|b'|^2B=|b'|^2(c'+\bar{c'})=(c'-\bar{c'}=0,$ i.e., $b'=0\in\mathbb{R},$ a contradiction. Hence, by lemma $(2.2)$, such system can be solved explicitly as claimed.
Consequently, $x=(-\Re b)/2-(b'+T)^{-1}(c'-T).$
\end{proof}
\newpage
\noindent\textbf{Example 1.}
Consider the equation $x^2+5x+6=0$, 
Find solutions $x\in\mathbb{O}$
\\
\textbf{Solution:}
\\
$b=5$ and $c=6$, therefore, $b^2>4c.$ From theorem 2.3, case 1.
\begin{equation*}
x=\frac{-5\pm\sqrt{5^2-4\times6}}{2}=\frac{-5\pm1}{2}= -3\hspace{2mm} \textrm{or} -2
\end{equation*}
\\
\textbf{Example 2.}
Consider the equation $x^2+1=0$,
Find solutions $x\in\mathbb{O}$
\\
\textbf{Solution:}
\\
$b=0$ and $c=1$, therefore, $b^2<4c.$ From theorem 2.3, case 2.
\begin{equation*}
x=\alpha \textbf{i}+\beta\textbf{j}+\gamma \textbf{k}+\delta\textbf{l}+p\textbf{il}+ q\textbf{jl}+r\textbf{kl}
\end{equation*}
\begin{equation*}
\hspace{1cm}\textrm{where} \ \    \alpha^2+\beta^2+\gamma^2+\delta^2+p^2+q^2+r^2=1
    \hspace{2mm} \textrm{and} \hspace{2mm} \alpha,\beta,\gamma,\delta,p,q,r\in\mathbb{R}
\end{equation*}
so there are infinitely many solutions to the equation $x^2+1=0$ in $\mathbb{O}$.
\\
\textbf{Example 3.} 
Consider the equation $x^2-x+\textbf{l}=0$, Find solutions $x\in\mathbb{O}$
\\
\textbf{Solution:}
\\
$b=-1$ and $c=\textbf{l}$,  From theorem 2.3, case 3.
\begin{equation*}
x=\frac{1}{2}\pm\frac{\rho}{2}\mp\frac{\textbf{l}}{\rho}
\end{equation*}
where
\begin{equation*}
    \rho=\sqrt{\frac{1+\sqrt{17}}{2}}
\end{equation*}
\textbf{Example 4.}
Consider the equation $x^2+\textbf{l}x+\textbf{j}=0$, Find solutions $x\in\mathbb{O}$
\\
\textbf{Solution:}
\\
$b=\textbf{l}\notin\mathbb{R}$ and $c=\textbf{j}$,  From theorem 2.3, case 4
\begin{equation*}
    x=-\frac{\Re b}{2}-(b'+T)^{-1}(c'-N)
\end{equation*}
since $\Re  b=0$ we have $x=-(b'+T)^{-1}(c'-N)$
\noindent where
\noindent note that $B^2<4E$ 
\\
Therefore,
$T=\pm\sqrt{2\sqrt{E}-B}=\pm 1$ and $N=\sqrt{E}=1$
\\
Thus,
$x=-(\textbf{l}+1)^{-1}(\textbf{j}-1)$ or $x=-(\textbf{l}-1)^{-1}(\textbf{j}-1)$
\\
Finally,
\begin{equation*}
x=\frac{(\textbf{l}-1)(1-\textbf{j})}{-2}=\frac{\textbf{l}+\textbf{jl}-1+\textbf{j}}{-2}=\frac{1-\textbf{j}-\textbf{l}-\textbf{jl}}{2}  
\end{equation*}
or
\begin{equation*}
x=\frac{(\textbf{l}+1)(1-\textbf{j})}{-2}=\frac{\textbf{l}+\textbf{jl}+1-\textbf{j}}{-2}=\frac{-1+\textbf{j}-\textbf{l}-\textbf{jl}}{2}  
\end{equation*}

\newpage
\section{Application}
\subsection{Left spectrum of $2\times 2$ octonionic matrix}
We can find left spectrum of $2\times 2$ octonionic matrices by solving corresponding octonionic quadratic equation.
Let's begin with following lemma
\begin{lemma} 
For $p$ and $q\in\mathbb{O},\ \sigma_l (pI+qA)=\{p+qt:t\in \sigma_l(A)\}$, where $I$ is the identity matrix.
\end{lemma}
\begin{proof}
Let $t$ be a left eigenvalue of $A$ with eigenvector $v$,
\begin{equation*}
    Av=tv
\end{equation*}
consider 
\begin{equation*}
    (pI+qA)v=pIv+qAv=pv+qtv=(p+qt)v
\end{equation*}
Therefore, $p+qt$ where $t\in\sigma_l(A)$ represent the set of left spectrum of $(pI+qA).$\
\end{proof}
\begin{theorem} Let
$A=
\begin{bmatrix} 
a & b \\
c & d 
\end{bmatrix}$, \ where $a,b,c, \in \mathbb{O}$.
\\
 (i) if $bc=0$, then $\sigma_l(A)=\{a,d\}$\\
 (ii) if $bc\neq 0$, then  $\sigma_l(A)=\{a+bt: t^2+b^{-1} (a-d)t-b^{-1}c=0\}$
\end{theorem}
\begin{proof} 
(i) $A$ is a triangular matrix, then results follows.\\
(ii) \textbf{Using Lemma 3.1}, we have
\begin{equation*}
\sigma_l(A)= \sigma_l \begin{bmatrix} aI+ b\begin{bmatrix}
0 & 1 \\
b^{-1}c & b^{-1}(d-a) 
 \end{bmatrix} \end{bmatrix}
\end{equation*}
Let $t$ be any left eigenvalue of 
$ \begin{bmatrix}
0 & 1 \\
b^{-1}c & b^{-1}(d-a) 
 \end{bmatrix}$, then there exists non-zero vector 
 $v=\begin{bmatrix} x\\
 y 
 \end{bmatrix}$
such that
\begin{equation*}
\begin{bmatrix}
 0 & 1 \\
b^{-1}c & b^{-1}(d-a)
 \end{bmatrix}\begin{bmatrix} x\\
 y 
 \end{bmatrix}=t\begin{bmatrix} x\\
 y 
 \end{bmatrix}
\end{equation*}
\begin{equation}
    y=tx,
  \end{equation}
  \begin{equation}
          b^{-1}cx+b^{-1}(d-a)y=ty
  \end{equation}
From equations (3.1) and (3.2), we have
\begin{equation}
    t^2+b^{-1}(a-d)t-b^{-1}c=0.
\end{equation}
We can use Theorem $(3.2)$ to find roots of $(3.3)$ and explicitly find left spectrum of given $2\times 2$ octonionic Matrix $A$.
\end{proof}



\end{document}